\documentclass[11pt]{article}
\usepackage{amsthm,amssymb,amsmath}
\usepackage{epsfig}
\usepackage[dvips]{color}

\newtheorem{prelem}{{\bf Theorem}}

 \newtheorem{theorem}{Theorem}
\newtheorem{corollary}[theorem]{Corollary}
\newtheorem{lemma}[theorem]{Lemma}
\newtheorem{observation}[theorem]{Observation}

\theoremstyle{definition}

\theoremstyle{remark}

\title{Lower Bounds on  Nonnegative Signed  Domination Parameters in Graphs}
\date{}
\author
{Arezoo. N. Ghameshlou\thanks{Research supported by the research council of
Faculty of Agriculture Engineering and Technology, University of
Tehran, through grant No. 322870/1/03.}\\
Department of Irrigation and Reclamation Engineering\\
University of Tehran, I.R. Iran\\
{\tt a.ghameshlou@ut.ac.ir}\vspace{3mm}\\
}
\begin{document}
\maketitle
\begin{abstract}
 \noindent Let $1 \leq k \leq n$ be a positive integer. A {\em nonnegative signed $k$-subdominating
function} is a function $f:V(G) \rightarrow \{-1,1\}$ satisfying
$\sum_{u\in N_G[v]}f(u) \geq 0$ for at least $k$ vertices $v$ of $G$.
The value $\min\sum_{v\in V(G)} f(v)$, taking over all
nonnegative signed $k$-subdominating functions $f$ of $G$, is called the {\em  nonnegative signed $k$-subdomination
number} of $G$ and denoted by $\gamma^{NN}_{ks}(G)$. When $k=|V(G)|$,
$\gamma^{NN}_{ks}(G)=\gamma^{NN}_s(G)$ is the  {\em  nonnegative signed domination
number}, introduced
in \cite{HLFZ}.  In this paper, we investigate several sharp lower bounds of $\gamma^{NN}_s(G)$, which extend some presented  lower bounds on $\gamma^{NN}_s(G)$. We also initiate the study of the nonnegative signed $k$-subdomination number in graphs and establish some sharp lower bounds for $\gamma^{NN}_{ks}(G)$ in terms of
order and the degree sequence of a graph $G$.
\vspace{3mm}

\end{abstract}

\section{Introduction}
Let $G$ be a simple graph of {\em order} $n$ with the vertex set $V(G)$ and {\em size} $m$ with the
edge set $E(G)$. We use \cite{W} for terminology and notation,
which are not defined here. The {\em minimum} and {\em maximum  degrees}
in graph $G$ are denoted by $\delta(G)$ and $\Delta(G)$,  respectively. A vertex $v \in V(G)$ is called an \textit{odd} (\textit{even}) vertex if  $\deg_G(v)$ is odd (even). For a graph $G=(V,E)$, let $V_{o}$ ( $V_e$) be the set of odd (even) vertices with $n_o=\mid V_o \mid$ and $n_e=\mid V_e \mid$. If $X \subseteq V(G)$, then $G[X]$ is the {\em subgraph} of $G$ {\em induced} by $X$. For disjoint subsets $X$ and $Y$ of vertices we let
$E(X,Y)$ denote the set of edges between $X$ and $Y$. The {\em
open neighborhood} $N_G(v)$ of a vertex $v\in V(G)$ is the set of
all vertices adjacent to $v$. Its {\em closed neighborhood} is
$N_G[v] = N_G(v) \cup \{v\}$.  In addition, the \textit{open} and \textit{closed neighborhoods} of a subset $S \subseteq V(G)$ are $N_G(S)=\cup_{v\in S}N_G(v)$ and
$N_G[S]=N_G(S)\cup S$, respectively. The \textit{degree} of a vertex $v \in V(G)$ is $\deg_G(v)= \mid N_G(v) \mid$.
For a real-valued function $f: V(G)
\longrightarrow {\mathbb R}$ the weight of $f$ is $\omega(f)=\sum
_{v \in V(G)}f(v)$ and for a subset $S$ of $V(G)$ we define $f(S)
= \sum_{v\in S} f(v)$, so $\omega(f)=f(V(G))$.
For a positive integer $1 \leq k \leq n$, a {\em signed $k$-subdominating function} (SkSDF) of G is a function $f: V(G)\longrightarrow \{-1, 1\}$ such that $f (N_G[v]) \geq1$ for
at least $k$ vertices $v$ of $G$. The {\em signed $k$-subdomination number} for a graph $G$ is defined as $\gamma_{ks}(G) = \min\{w(f )|f \mbox{ is a SkSDF of } G\}$. The concept of the signed $k$-subdomination number was introduced and studied by Cockayne and Mynhardt \cite{CM}.
A {\em nonnegative signed dominating function} (NNSDF) of $G$  defined in
\cite{HLFZ} as a function  $f : V(G) \longrightarrow \{-1, 1\}$ such
that $f(N_G[v]) \geq 0$ for all vertices $v$ of $G$. The {\em nonnegative signed
domination number} (NNSDN) of $G$ is
$\gamma^{NN}_{s}(G)= \min\{\omega(f)| f \mbox{ is an NNSDF of } G \}$.

We now introduce a
{\em nonnegative signed $k$-subdominating function} (NNSkSDF) of $G$ for a positive integer $1 \leq k \leq n$ as a function $f :
V(G) \longrightarrow \{-1, 1\}$ such that $f(N_G[v]) \geq 0$ for at
least $k$ vertices $v$ of $G$. We define the {\em nonnegative signed  $k$-subdomination number}
(NNSkSDN) of $G$ by $\gamma^{NN}_{ks}(G)= \min\{\omega(f)\mid f \mbox{ is a NNS}k\mbox{SDF of  }G \}$.  A nonnegative signed $k$-subdominating function of weight $\gamma^{NN}_{ks}(G)$ is called a $\gamma^{NN}_{ks}(G)$-{\em function}. Note that
$\gamma^{NN}_{ns}(G)=\gamma^{NN}_{s}(G)$. Since every signed $k$-subdominating function of $G$ is a nonnegative signed $k$-subdominating function, we deduce that
$$\gamma^{NN}_{ks}(G) \leq \gamma_{ks}(G).$$
For a function $f : V(G)
\longrightarrow \{-1, 1\}$ of $G$ we define $P=\{v\in V(G)\mid
f(v)=1\}$,  $M=\{v\in V(G)\mid f(v)=-1\}$, and $C_f=\{v \in V(G)| f (N_G[v])\geq 0 \}$.

In this paper, we establish some new  lower bounds on $\gamma^{NN}_s(G)$  for a general
graph in terms of various different graph parameters. Some of these bounds improve several lower bounds on $\gamma^{NN}_s(G)$ presented in \cite{AS,HLFZ}. We also
initiate the study of nonnegative signed $k$-subdomination numbers in graphs, and present some sharp lower bounds for $\gamma^{NN}_{ks}(G)$ in terms of the
order and the degree sequence of a graph $G$.
\begin{observation}\label{ob}
{\rm Let $f$ be an NNSkSDF of $G$. For $v \in C_f$ if $v$ is an even vertex, then $f(N_G[v]) \geq 1$ while $f(N_G[v]) \geq 0$ if  $v$ is an odd vertex. }
\end{observation}
\begin{observation}\label{Signed}
{\rm Let $1\leq k\leq n$ be a positive integer. For any even graph $G$, $$\gamma^{NN}_{ks}(G)=\gamma_{ks}(G).$$ }
\end{observation}
In this paper, we make use of the following results.
\begin{prelem}\cite{AS}\label{Ata3}
{\rm Let   $G$ be a graph of order $n$ and size $m$. Then
$$\gamma^{NN}_{s}(G) \geq \frac{n}{2}-m.$$
}
\end{prelem}

\begin{prelem}\cite{AS}\label{Ata5}
{\rm Let   $G$ be a graph of order $n$, size $m$ and minimum degree $\delta$. Then
$$\gamma^{NN}_{s}(G) \geq \frac{-4m+3n\lceil\frac{\delta+1}{2}\rceil-n}{3\lceil\frac{\delta+1}{2}\rceil+1}.$$
}
\end{prelem}

\begin{prelem}\cite{DHHS}\label{cycle}
{\rm For $n \geq 3$, $$\gamma_s(C_n)=\left\{
\begin{array}{lcl}
n/3  & \mbox{if} &  n \equiv 0 \pmod  3, \\
\lfloor\frac{n}{3}\rfloor+1 & \mbox{if} & n \equiv 1 \pmod  3, \\
\lfloor\frac{n}{3}\rfloor+2 & \mbox{if} & n \equiv 2 \pmod  3. \\
\end{array} \right. $$ }
\end{prelem}
\begin{corollary}\cite{H}\label{Hen}
{\rm For any  $r$-regular graph $G$ of order $n$,
$\gamma_{s}(G) \geq \dfrac{n}{r+1}$, for $r$ even.
Furthermore this bound is sharp.}
\end{corollary}

\begin{prelem}\cite{HLFZ}\label{Hua6}
{\rm Let $K_n$ be a complete graph. Then  $\gamma^{NN}_{s}(K_n) =0$ when $n$ is even and $\gamma^{NN}_{s}(K_n) =1$ when $n$ is odd.
}
\end{prelem}
\begin{prelem}\cite{HLFZ}\label{Hua10}
{\rm For any graph $G$ with maximum degree $\Delta$ and minimum degree $\delta$, we have $$\gamma^{NN}_{s}(G) \geq \frac{\delta-\Delta}{\delta+\Delta+2}n.$$
}
\end{prelem}
\section{Lower bounds on the  NNSDNs of graphs }
In this section, we present some new sharp lower bounds for $\gamma^{NN}_{s}(G)$  by using $n_e$ as  the number of  even vertices in a graph $G$.  We begin with the following lemma.

\begin{lemma}\label{lemNN}

{\rm Let $f$ be an  NNSDF of a simple connected graph $G$. Then,

\begin{enumerate}
\item $\sum_{v\in P} \deg_{G}(v) \geq  n+n_e-2\mid P \mid +\sum_{v \in M} \deg_{G}(v)$.
\item $\sum_{v \in P}\deg_{G[P]}(v)\geq \sum_{v \in
P}\lceil \frac {\deg_{G}(v)-1}{2}\rceil$.
\end{enumerate}
}
\end{lemma}
\begin{proof}
For $v \in V(G)$, let $\deg_P(v)$ and $\deg_M(v)$ denote the numbers of vertices of $P$ and $M$, respectively, which are adjacent
to $v$.  Clearly, $\deg_{G}(v)=\deg_{M}(v)+\deg_{P}(v)$.
Since $f(N_G[v]) \geq 0$,  for every $v \in P$, $\deg_{M}(v) \leq
\deg_{P}(v)+1$, and for every  $v \in M$,   $\deg_{P}(v) \geq
\deg_{M}(v)+1$. Hence, if $v \in P$, then
$\deg_{M}(v) \leq \lfloor \frac{\deg_{G}(v)+1}{2}\rfloor$ and if $v
\in M$, then $\deg_{P}(v) \geq \lceil
\frac{\deg_{G}(v)+1}{2}\rceil$.
\begin{enumerate}
\item Counting the number of edges in $E(P,M)$
in two ways, we can deduce that
$$\begin{array}{ccccl}
 \sum_{v \in M} \lceil \frac{\deg_{G}(v)+1}{2} \rceil & \leq &  |E(P,M)| & \leq &
 \sum_{v \in P} \lfloor \frac{\deg_{G}(v)+1}{2} \rfloor
  \end{array}$$
\noindent It follows that
$$\begin{array}{ccl}
\sum_{v \in P}\deg_{G}(v)+ \mid P \mid & \geq & \sum_{v \in V} \lceil\frac{\deg_{G}(v)+1}{2}\rceil\\\\
 &=&\sum_{v \in V_o}\frac{\deg_G(v)+1}{2}+\sum_{v \in V_e}\frac{\deg_G(v)+2}{2}\\\\
 &=&\sum_{v \in V}\frac{\deg_G(v)}{2}+n_e+\frac{n_o}{2}\\\\
  &=&\sum_{v \in P}\frac{\deg_G(v)}{2}+\sum_{v \in M}\frac{\deg_G(v)}{2}+n_e+\frac{n_o}{2},\\
\end{array}$$

which implies that
$$\sum_{v \in P}\frac{\deg_G(v)}{2} \geq  \sum_{v \in M}\frac{\deg_G(v)}{2}+n_e+\frac{n_o}{2}-\mid P\mid.$$ Hence,
$$\begin{array}{ccl}
 \sum_{v\in P} \deg_{G}(v)& \geq & \sum_{v \in M} \deg_{G}(v)+n+n_e-2\mid P \mid.\\
\end{array}$$
\item Consider the subgraph $G[P]$ induced by $P$. We have
$\deg_{G[P]}(v) = \deg_{P}(v)$ for each $v \in P$. Since
$\deg_P(v) \geq \lceil \frac {\deg_{G}(v)-1}{2}\rceil$ for each $v
\in P$, we have
$$\sum_{v \in P}\deg_{G[P]}(v)\geq \sum_{v \in P}\lceil \frac
{\deg_{G}(v)-1}{2}\rceil.$$
\end{enumerate}
\end{proof}
In the next theorem we  present  some lower bounds on
$ \gamma^{NN}_{s}(G)$. By using Lemma \ref{lemNN} and graph parameters such as order, size, number of even vertices, maximum and minimum degrees
we obtain some new lower bounds for $\gamma^{NN}_{s}(G)$. These new results are independent from each other.
\begin{theorem}\label{theoNN}
{\rm Let $G$ be a simple connected graph of
order $n$, minimum degree $\delta$, maximum degree $\Delta$ and the number of even vertices $n_e$.
Then
\begin{enumerate}
\item $\gamma^{NN}_{s}(G) \geq \dfrac{n\delta-n\Delta+2n_e}{\Delta+\delta+2}$,
\\
\item  $\gamma^{NN}_{s} (G) \geq \dfrac{2m+n_e-n\Delta}{\Delta+1}$,
\\
\item $\gamma^{NN}_{s} (G)\geq \dfrac{n\delta+n_e-2m}{\delta+1} $,
\\
\item  $\gamma^{NN}_{s} (G)\geq \lceil\dfrac{-(\delta+1)+ \sqrt{(\delta+1)^2 +8(n\delta +n+n_e)}}{2}-n\rceil$,
\\
\item  $\gamma^{NN}_{s} (G) \geq \lceil\sqrt{2m+n+n_e}-n\rceil$.
\end{enumerate}
}
\end{theorem}
\begin{proof}
According to Lemma \ref{lemNN}, we have
\begin{equation}\label{eq1}
\sum_{v \in M}\deg_G(v)+n+n_e-2 \mid P \mid \leq \sum_{v \in P}\deg_G(v).
\end{equation}
\begin{enumerate}
\item  \label{1} Since  $\delta \leq \deg_G(v) \leq \Delta$ for each $v \in V(G)$,  inequality (\ref{eq1}) follows that
$$\delta n- \mid P \mid\delta+n+n_e-2\mid P \mid\leq \sum_{v\in P}\deg_{G}(v) \leq \Delta\mid P\mid.$$

From this inequality, it is deduced that
$$\mid P \mid \geq \frac{\delta n +n+n_e}{\Delta+\delta+2}.$$

Hence, $$\gamma^{NN}_{s} (G)=2\mid P \mid-n \geq
\dfrac{n\delta-n\Delta+2n_e}{\Delta+\delta+2}.$$
\item \label{2} Since $2m=\sum_{v \in V} \deg_G(v)$ and
$\deg_G(v) \leq \Delta$ for each $v \in V(G)$, by inequality
(\ref{eq1}) it follows that
$$\begin{array}{ccl}
 2m+n+n_e-2\mid P \mid & = & \sum_{v \in V}\deg_{G}(v)+n+n_e-2\mid P \mid \\\\
           & \leq & 2\sum_{v \in  P}\deg_{G}(v)\\\\
           & \leq & 2\Delta\mid P \mid.
\end{array}$$
Therefore, $2\mid P \mid\geq \dfrac{2m+n+n_e}{\Delta+1}$, and
 $\gamma^{NN}_{s} (G) \geq \dfrac{2m+n_e-n\Delta}{\Delta+1}$, as desired.
 \item  \label{3} Using inequality (\ref{eq1}) and the facts  $2m=\sum_{v \in V}
\deg_G(v)$ and $\deg_G(v) \geq \delta$ for any $v \in V(G)$, we
have
$$\begin{array}{ccl}
2m & = & \sum_{v \in V}\deg_{G}(v)\\\\
  & \geq & 2\sum_{v \in M}\deg_{G}(v)+n+n_e-2\mid P \mid\\\\
& \geq & 2n\delta-2\delta\mid P \mid+n+n_e-2\mid P \mid
\end{array}$$

It follows that
$$2\mid P \mid\geq \dfrac{2n\delta +n+n_e-2m}{\delta+1}.$$

Thus, $\gamma^{NN}_{s} (G)\geq \dfrac{n\delta+n_e-2m}{\delta+1} $, as desired.
\item \label{4} Consider $G[P]$. According to Lemma
\ref{lemNN}, we have
$$\sum_{v \in P}\deg_{G[P]}(v)\geq \sum_{v \in
P}\lceil \dfrac {\deg_{G}(v)-1}{2}\rceil \geq \sum_{v \in P} \dfrac{\deg_{G}(v)-1}{2} .$$ On the other hand,
since $G[P]$ is a simple connected graph,  $$\sum_{v \in P}\deg_{G[P]}(v)\leq
\mid P \mid(\mid P \mid-1).$$ Thus,
$$\begin{array}{ccl}
2\mid P \mid(\mid P \mid-1) &\geq& \sum_{v \in P}\deg_{G}(v) -\mid P \mid\hfill\\\\
                   & \geq & \sum_{v \in M}\deg_{G}(v)+n+n_e-3\mid P \mid\hfill\\\\
                   & \geq & n\delta-\delta\mid P \mid+n+n_e-3\mid P\mid. \hfill\\\\
                   \end{array}$$
This implies that
$$2\mid P \mid^2+(\delta+1)\mid P \mid-(n\delta+n+n_e)\geq 0.$$
Therefore,
$$2\mid P \mid \geq \frac{-(\delta+1)+ \sqrt{(\delta+1)^2 +8(n\delta +n+n_e)}}{2},$$

and hence  $\gamma^{NN}_{s} (G)\geq \lceil \dfrac{-(\delta+1)+ \sqrt{(\delta+1)^2 +8(n\delta +n+n_e)}}{2}-n\rceil$, as desired.
\item \label{5} By Parts (\ref{4}) and (\ref{2}) we have
$$2\sum_{v \in P}\deg_G(v) \leq 4\mid P \mid^2-2\mid P \mid, $$
and $$2\sum_{v \in P}\deg_G(v) \geq 2m+n+n_e-2\mid P \mid,$$
respectively.  So,
$$4\mid P \mid^2\geq  2m+n+n_e,$$
which implies that
$$ \mid P \mid\geq \frac{\sqrt{2m+n+n_e}}{2}.$$
Thus, $\gamma^{NN}_{s} (G)\geq \lceil\sqrt{2m+n+n_e}-n\rceil $,  as
desired.
\end{enumerate}
\end{proof}
From Theorem \ref{theoNN} (\ref{1})$-$(\ref{3}), we have the
following result. For $k=n$ by Observation \ref{Signed} when $r$ is even, we have the same bound presented in Corollary \ref{Hen} by Henning.
\begin{corollary}\label{regular}
{\rm For $r\geq 1$, if $G$ is an $r$-regular graph of order $n$,
then
$$\gamma^{NN}_{s}(G) \geq \left\{\begin{array}{ll}
\dfrac{n}{r+1} & \mbox{ if } $r$ \mbox{ is even,} \\
&\\
0 & \mbox{ if } $r$ \mbox{ is odd}. \\
\end{array}\right.$$
Furthermore, these bounds are sharp.}
\end{corollary}
\noindent In order to show that the bounds presented in Theorem \ref{theoNN} are sharp, we will give a graph $G$ and construct a $\gamma^{NN}_s(G)$-function $f$
such that $w(f)$ achieves the lower bounds, and thus the lower bounds are sharp. We also illuminate that our bounds for
some of these graphs are attainable while the corresponding bounds given in Theorems   \ref{Ata3}, \ref{Ata5}, and  \ref{Hua10} are not.
In fact, a trivial examples such  $G \in \{K_n,C_n\}$ is sufficient for this. It is easy to
see that $\gamma^{NN}_s(K_n)$ obtains all the five bounds in Theorem \ref{theoNN} while the bound in Theorem \ref{Ata3} shows that $\gamma^{NN}_s(K_n) \geq \frac{2n-n^2}{2}$  and the bound in Theorem  \ref{Ata5} is not more than  $\frac{5n-n^2}{3n+5}$. As an other example, $\gamma^{NN}_s(C_n)$ attains the lower bounds in
Theorem \ref{theoNN} (\ref{1})$-$(\ref{3}), when $n\equiv 0\;\; (mod\;\;3)$ while the bounds in Theorems \ref{Ata3},  \ref{Ata5} and \ref{Hua10} are
not more than $\frac{n}{7}$. Besides, we can construct a non-complete graph with an arbitrary large order whose  reaches the lower bounds in Theorem \ref{theoNN} (\ref{1})$-$(\ref{3}) as follows. Letting $t$ be a positive integer, we consider a cycle of length $2t$.  For every  edge, we include an additional vertex being adjacent to both endpoints
of the corresponding edge. The obtained graph is denoted by $G$.   It is easy to check that the graph $G$ is a graph with $n=4t$, $m=6t$, $\delta=2$, $\Delta=4$ and $n_e=4t$.
Define a function $f : V(G) \longrightarrow  \{-1, 1\}$ as follows: $f (v)=1$ for $v \in V(C_{2t})$ and $f (v)=-1$ for
 $v \in V(G) \setminus V(C_{2t})$. It is easy to verify that $f$ is a $\gamma^{NN}_s(G)$-function with
$w(f ) = 0$ and  bounds in Theorem \ref{theoNN} (\ref{1})$-$(\ref{3})  are also $0$, which implies that $G$ is sharp for these bounds. However, $\gamma^{NN}_s(G)$ does not attain the corresponding bounds given in Theorems \ref{Ata3}, \ref{Ata5} \ and \ref{Hua10}, which are $-4t$,  $\lceil-\frac{4t}{7}\rceil$, and $-t$, respectively. Next, we show that there is also a graph $G$ different from $K_n$ such that $\gamma^{NN}_s(G)$ reaches  lower bounds in
Theorem \ref{theoNN} (\ref{4})$-$(\ref{5}). Let $H$ be the Haj\'{o}s graph. We can verify that $\gamma^{NN}_s(H)=0$, and  $H$ is sharp for presented bounds in Theorem \ref{theoNN} (\ref{4})$-$(\ref{5}).
\section{Lower  bounds on the NNSkSDNs of graphs}
In this section, we initiate the study of  the nonnegative signed $k$-subdomination number in graphs. we present some lower bounds on the nonnegative signed $k$-subdomination number  of a graph in terms of the order and the degree sequence. We begin with the following lemma.

\begin{lemma}\label{lemkNN}
{\rm Let $G$ be a graph and $1\leq k \leq n$ be a positive integer. Let $f$ be a $\gamma^{NN}_{ks}(G)$-function.
 Let $P_1=P \cap C_f$, $P_2=P \setminus P_1$, $M_1=M \cap C_f$ and $M_2 =M\setminus M_1$. Then,
$$\sum_{v \in P}\deg_G(v)+\mid P_1 \mid \geq \sum_{v \in P_1 \cup M_1}\lceil \dfrac{\deg_G(v)+1}{2}\rceil.$$
}
\end{lemma}
\begin{proof}
Note that if $v \in V(G)$, then $\deg_G(v)=\deg_P(v)+\deg_M(v)$.  For $v \in P_1 \cup M_1$, $f (N_G[v])\geq0$. So, if  $v \in P_1$, then $\deg_P(v) \geq \lceil\frac{\deg_G(v)-1}{2}\rceil$ and $\deg_M(v) \leq \lfloor\frac{\deg_G(v)+1}{2}\rfloor$. Similarly, if $v \in M_1$, then $\deg_P(v) \geq \lceil\frac{\deg_G(v)+1}{2}\rceil$ and $\deg_M(v) \leq \lfloor\frac{\deg_G(v)-1}{2}\rfloor$.

Counting the number of edges in $E(P,M)$
in two ways, we  conclude that
$$\sum_{v \in M_1}\lceil\frac{\deg_G(v)+1}{2}\rceil+\sum_{v \in M_2}\deg_P(v) \leq \sum_{v \in P_1}\lfloor\frac{\deg_G(v)+1}{2}\rfloor+\sum_{v \in P_2}\deg_M(v).$$
By adding $\sum_{v \in P_1}\lceil\frac{\deg_G(v)+1}{2}\rceil$ to the both sides of the inequality we have
$$\begin{array}{ccc}
\sum_{v \in P}\deg_G(v)+\mid P_1 \mid &\geq& \sum_{v \in P_1 \cup M_1}\lceil\dfrac{\deg_G(v)+1}{2}\rceil+ \sum_{v \in M_2}\deg_G(v)\hfill\\\\
&\geq& \sum_{v \in P_1 \cup M_1}\lceil\dfrac{\deg_G(v)+1}{2}\rceil,\hfill\\
\end{array}$$
\noindent and this completes the proof.
\end{proof}
\begin{theorem}\label{theokNN}
{\rm For any graph $G$ with degree sequence $d_1\leq d_2 \leq \cdots \leq d_n$ and any positive integer $1 \leq k \leq n$,
\begin{enumerate}

\item[1.] $\gamma^{NN}_{ks}(G) \geq \dfrac{2\sum_{i=1}^{k}\lceil \frac{d_i+1}{2}\rceil}{\Delta+1}-n$.

\item[2.] $\gamma^{NN}_{ks}(G) \geq \dfrac{n\delta-4m-n+2\sum_{i=1}^{k}\lceil \frac{d_i+1}{2}\rceil}{\delta+1}$.\\
Furthermore, these bounds are sharp.
\end{enumerate}
}
\end{theorem}
\begin{proof}
Considering Lemma \ref{lemkNN} we have
\begin{equation}\label{eq31}
\sum_{v \in P}\deg_G(v)+\mid P_1 \mid \geq \sum_{v \in P_1 \cup M_1}\lceil \dfrac{\deg_G(v)+1}{2}\rceil.
\end{equation}

\begin{enumerate}
\item \label{k1} Since $\delta \leq \deg_G(v) \leq \Delta$ for each $v \in V(G)$, inequality (\ref{eq31}) follows that
$$\Delta \mid P \mid+\mid P \mid \geq \sum_{v \in P}\deg_G(v)+\mid P_1 \mid \geq \sum_{v \in P_1 \cup M_1} \lceil\dfrac{\deg_G(v)+1}{2}\rceil.$$
Note that  $P_1 \cap M_1=\emptyset$ and $\mid P_1 \cup M_1 \mid=\mid P_1 \mid+\mid M_1 \mid \geq k$. So,
$$\mid P\mid \geq \dfrac{\sum_{v \in P_1 \cup M_1} \lceil\dfrac{\deg_G(v)+1}{2}\rceil}{\Delta+1} \geq \dfrac{\sum_{i=1}^k\lceil\dfrac{d_i+1}{2}\rceil}{\Delta+1}.$$
Thus,
$$\gamma^{NN}_{ks}(G)=2 \mid P \mid-n \geq \dfrac{2\sum_{i=1}^{k}\lceil \frac{d_i+1}{2}\rceil}{\Delta+1}-n.$$
\item \label{k2} Obviously, , $2m=\sum_{v \in V(G)}\deg_G(v)=\sum_{v \in P}\deg_G(v)+\sum_{v \in M}\deg_G(v)$. If we add $\mid P_1 \mid$ to the both sides of this equality, then by Lemma \ref{lemkNN} we deduce that
$$\begin{array}{ccc}
\mid P \mid \geq \mid P_1 \mid &\geq& -2m+ \sum_{v \in P_1 \cup M_1}\lceil \dfrac{\deg_G(v)+1}{2}\rceil+\sum_{v \in M}\deg_G(v)\hfill\\\\
&\geq& -2m+\sum_{i=1}^{k}\lceil\dfrac{d_i+1}{2}\rceil+\delta n - \delta \mid P\mid.\hfill\\
\end{array}$$
Therefore,
$$\mid P \mid \geq \dfrac{n\delta-2m+\sum_{i=1}^{k}\lceil \frac{d_i+1}{2}\rceil}{\delta+1},$$
and hence,
$$\gamma^{NN}_{ks}(G)=2\mid P \mid -n \geq  \dfrac{n\delta-4m-n+2\sum_{i=1}^{k}\lceil \frac{d_i+1}{2}\rceil}{\delta+1}.$$
\end{enumerate}
Now suppose that $k =n$, considering that $2\sum_{i=1}^n \lceil \dfrac{d_i+1}{2}\rceil =2m+n+n_e$, we can immediately obtain those two bounds in Theorem \ref{theoNN} (\ref{2}) and (\ref{3}) from the lower  bounds of  Theorem
\ref{theokNN}, respectively. Since the bounds in Theorem \ref{theoNN} are sharp, so there exist graphs whose $\gamma^{NN}_{ks}(G)$  recieve the bounds in
Theorem \ref{theokNN}. Therefore, these  bounds  are sharp.
\end{proof}
\noindent As an immediate consequence of Theorem \ref{theokNN} we have the
following result.
\begin{corollary}\cite{HATT}\label{Hatt}
{\rm For every $r$-regular graph $G$ of order $n$,  $\gamma_{ks}(G)\geq   \dfrac{k(r + 2)}{r + 1}-n$ for $r$ even. }
\end{corollary}
\begin{corollary}\label{kregular}
{\rm For $r\geq 1$, if $G$ is an $r$-regular graph of order $n$,
then
$$\gamma^{NN}_{ks}(G) \geq \left\{\begin{array}{ll}
\dfrac{k(r+2)}{r+1}-n & \mbox{ if } $r$ \mbox{ is even,} \\
&\\
k-n & \mbox{ if } $r$ \mbox{ is odd}. \\
\end{array}\right.$$
Furthermore, these bounds are sharp.}
\end{corollary}
\noindent Clearly, if $r$ is even, then  by Observation \ref{Signed}  we have the same given bound in Corollary \ref{HATT}.


\end{document}